\documentclass[12pt,reqno,a4paper]{amsart}
\usepackage{extsizes}
\usepackage{hyperref}
\usepackage{amsmath,amsfonts,enumerate,amssymb}
\usepackage{mathtools}
\usepackage{tikz-cd}
\usepackage{xcolor}
\usepackage{amsmath}
\usepackage[dvipsnames]{xcolor}

\newcommand{\R}{\mathbb{R}}

\newcommand{\Cu}{\text{Cu}}

\newcommand{\coz}{\text{Coz}}
\newcommand{\qt}{\text{QT}}
\newcommand{\T}{\text{T}}
\newcommand{\Aff}{\mathrm{Aff}}
\newcommand{\LAff}{\mathrm{LAff}}
\newcommand{\K}{\mathcal K}

\newcommand{\eps}{\varepsilon}

\newcommand{\Prob}{\mathrm{Prob}}

\newtheorem{thm}{Theorem}[section]
\newtheorem{lem}[thm]{Lemma}
\newtheorem{prop}[thm]{Proposition}
\newtheorem{cor}[thm]{Corollary}

\newtheorem{rmk}[thm]{Remark}

\newcommand{\rk}{\operatorname{rank}}
\newtheorem{innercustomthm}{Theorem}
\newenvironment{customthm}[1]
  {\renewcommand\theinnercustomthm{#1}\innercustomthm}
  {\endinnercustomthm}

\newtheorem{defn}[thm]{Definition}

\title{New structure for Villadsen algebras \\ of the first type }
\author{Alejandro Cano}
\address{Department of Mathematics, Purdue University, 150 N University St, West Lafayette IN 47907, USA}
\email{cano13@purdue.edu}
\author{Andrew S. Toms}
\address{Department of Mathematics, Purdue University, 150 N University St, West Lafayette IN 47907, USA}
\email{atoms@purdue.edu}

\begin{document}

\begin{abstract}
In this paper we establish two main results concerning the comparison theory of positive elements in approximately homogeneous $C^{*}$-algebras with diagonal maps. The first is an existence theorem which interfaces with the rank problem; given a Villadsen algebra $A$, we construct a convex cone that induces an embedding of a semigroup into $Cu(A)$. This semigroup is identified with  the set of strictly positive, lower-semicontinuous affine maps on the trace simplex of $A$. The second is a construction of an AH algebra that produces an obstacle to classifying the approximate unitary orbits of positive elements in terms of their spectral data. Our counterexample sheds light on a critical topological obstruction when attempting to understand equivalence classes of $*$-homomorphisms from $C_{0}((0,1])$ to an arbitrary $C^{*}$-algebra of stable rank one.\end{abstract}
\maketitle

\section{Introduction}

For nearly 30 years, the classification program has made use of Villadsen’s algebras as the prototypical testing ground for the predisposition of $C^{*}-$algebras to be classified via algebraic invariants. These algebras imposed restrictions upon the class of $C^{*}$-algebras which were once believed to be classified by the Elliott invariant; consisting of the ordered K-theory, trace simplex, and the pairing. The tensorial absorption of the Jiang-Su algebra, slow-dimension growth, and strict comparison were all various traits which Villadsen’s examples realized as equivalent regularity properties. The recurring appearance of these properties led to the formulation of the well-known conjecture named after the second author and Wilhelm Winter. 

In the last decade, the focus of the Elliott program shifted to the classification of $*$-homomorphisms between $C^{*}$-algebras, now with the use of ordered total K-theory, and traces \cite{carrion2023classifyinghomomorphismsiunital}, particularly maps $A\to B$ where $A$ is nuclear, and $B$ is $\mathcal{Z}$-stable. This line of investigation raises questions about general target algebras which need not be well-behaved.  The two considerations in this context deal with existence of liftings of maps at the level of the invariant in question, as well as the issue of uniqueness of the lifting. In a recent paper by Wan and the second author\cite{toms2025selfadjointoperatorszstablecalgebras}, the existence of $*$-homomorphisms from $C_{0}((0,1])$ to a pure $C^{*}-$algebra has been characterized as liftings of continuous affine maps from the quasitrace simplex into the set of probability measures on $[0,1]$. 
A related avenue of study in the classification program has been the ``Rank problem"; it is well-known that each fixed positive element induces a lower-semicontinuous affine function on the space of quasitraces of the $C^*$-algebra. What has been a subject of study by many is proving the converse, namely that for each function with such qualities, there exists a positive element whose rank function realizes this map. The work in this direction in recent times can be found in \cite{AntoinePerera2022},\cite{PereraToms2011}, and \cite{DADARLAT20101209}. While confirming an answer to the rank problem is essential in understanding the scope of the recently confirmed regularity conjecture\cite{evington2026regularsimplenuclearcalgebras}, we seek to demonstrate that further structure is present in $C^*$-algebras, even if they are considered poorly behaved. This leads us to our first result of the paper:

\begin{customthm}{1}
\label{main_result_1}
    Suppose that $A=\varinjlim (A_{i},\varphi_{i})$ is a simple, type I Villadsen algebra, each $A_{i}=M_{r_{i}}\otimes C(X^{m_{i}})$, each $X^{m_i}$ is compact and metrizable with diagonal eigenvalue maps, negligible point evaluations, and $r_i\to \infty$. There is a subsemigroup of $Cu(A)$ denoted $\text{Rksp}(A)$ and it is $\mathbf{Cu}$-isomorphic to $\{0\}\sqcup\mathrm{LAff}(T(A))_{++}$.
\end{customthm}

It has already been shown by \cite{AntoinePerera2022} that in a separable $C^*$-algebra of stable rank one with no finite-dimensional representations, the realization of ranks is possible. We build upon this evidence to show that indeed, all rank functions are present, and further, there exists a sub-semigroup which encodes the structure seen in the Cuntz semigroup of a pure $C^*$-algebra. We have proved this fact in the context of an AH algebra with diagonal maps which, while less broad than the set of all stable rank one algebras, has a Cuntz semigroup that has yet to be computed, thus our result is a first step in understanding the fine structures of the invariant that governs the classification of these algebras.\\

Since all $*-$homomorphisms between $C_{0}((0,1])$ and a $C^{*}$-algebra $A$ are in bijection with positive contractions $a\in A$ via the spectral theorem, determining the uniqueness of a lifting of this type amounts to the study of the distance between the unitary orbits of two positive elements in a given $C^{*}$-algebra. 

Investigating metrics on positive operators on a Hilbert space $\mathcal{H}$ which help determine the distance between their unitary orbits dates back to the work of Hermann Weyl, where he established an equivalence of metrics on positive-definite matrices: namely the matching distance between their spectra counted with multiplicity, and the distance between their unitary orbits $d_U$. Later in 1984 Azoff and Davis' theorem generalized Weyl's result to self-adjoint operators acting on a finite dimensional Hilbert space. In the last four decades, there have been significant extensions of this fundamental idea, see \cite{hiai1989distance} for the theorem in the context of a von Neumann algebra, \cite{JACELON20153304} in the context of a pure $C^*$-algebra, and \cite{JACELON2021109068} where KK-theory helps tie the metrics together.\\

The reach of these results is very broad, and allows one to estimate or in certain cases compute directly the distance of the unitary orbits of two positive elements in terms of their spectral data in large classes of $C^*$-algebras. This furnishes a natural question: ``Can tensorial absorption of the Jiang-Su algebra be removed in order to expand the class of algebras for which this is possible?'' The interest in the relations among the metrics, $d_T$, or $d_P$ and $d_U$ is a matter of being able to simplify the comparison theory of positive elements. Villadsen’s examples once again emerge as the natural test class.\\ 

In order to understand what is meant by spectral data of an operator, we introduce the notion of a \emph{spectral profile}, which is central in \cite{toms2025selfadjointoperatorszstablecalgebras}. \\

\textbf{Definition} \ref{def2.26}: Let $A$ be a unital $C^*$-algebra. For every positive contraction $a\in A$, we define the spectral profile $f_a:T(A)\to \text{Prob}([0,1])^+$ by $f_a(\tau)=\mu_{\tau,a}$ where $\mu_{\tau,a}$ is the Radon measure induced by $\tau$ on $\sigma(a)$ by the Riesz-Markov-Kakutani representation theorem. Our second main result is the following: 

\begin{customthm}{2}
\label{main_result_2}
    There exists a simple, separable, AH $C^*$-algebra $A$ and two positive elements $a$ and $b$ in $A$ that are not approximately unitarily equivalent, yet $f_a(\tau) =f_b(\tau)$ for all $\tau\in T(A)$.
\end{customthm}

We restate this in more familiar terms to illustrate its importance for classification. 

There exist $a,b\in A_+$, two positive contractions which both have spectrum equal to $[0,1]$. By functional calculus, these are in correspondence with two unital $*$-monomorphisms $\varphi_a,\varphi_b:C[0,1]\to A$. Since $[0,1]$ is contractible, $K_0(C[0,1])\cong \mathbb{Z}$, and $K_1(C[0,1])\cong 0$. These two maps induce the same map at the level of K-theory. Our construction implies that both homomorphisms induce the same map on traces, for every $h\in C([0,1])$ and $\tau\in T(A)$, \begin{align}
    \tau(\varphi_a(h))=\tau(\varphi_b(h))
\end{align}
Despite their agreement on these invariants, $\varphi_a\nsim \varphi_b$.\\

Our construction in theorem \ref{main_result_2} follows that of the second author's paper \cite{TomsAnnals}, where we consider a Villadsen algebra with seed space containing two 2-spheres; we then leverage the fact that there are nontrivial vector bundles present and use them to generate positive elements of equal ranks. The functions $f_a, f_b$ are continuous affine maps from the trace simplex into the set of regular Borel probability measures on the unit interval\cite{rudin86real}. At each point (trace), these maps can provide the generalized rank of the positive elements $a$ and $b$ respectively. If purity is added as a condition on the target algebra, then we obtain a ``collapsing'' of the two distinct Cuntz classes onto the single class that depends solely on the generalized rank. This is also true for the algebras in \ref{main_result_1}. Thus, the semigroup $\text{LAff}(T(A),(0,\infty])$ acts as a spine onto which all rank functions collapse once the algebra has been tensored with $\mathcal{Z}$.

\textbf{AI statement:} An LLM was used for proofreading and reference checking. The mathematical content of the paper is due solely to the authors.

\section{Preliminaries}
\subsection{Comparison of positive elements.}\hspace{1in}\\
We refer the reader to the surveys \cite{ara2009ktheoryoperatoralgebrasclassification} and \cite{GardellaPerera} for the basic theory of the Cuntz semigroup and dimension functions.  For the convenience of the reader we do at least recall here that for a unital C$^*$-algebra $A$ and normalized quasitrace $\tau:A \to \mathbb{C}$, the associated dimension function on $A_+$ is given by
\begin{equation*}
 d_\tau(a) := \lim_{n \to \infty} \tau\left(a^{1/n}\right)
\end{equation*}

\begin{prop}\cite{GardellaPerera}\label{gardella-dim-fns} 
Let $A$ be a unital C$^*$-algebra. Given $\tau \in QT(A)$, the map $d_\tau$
is well-defined on $\Cu(A)$. Moreover,
\begin{enumerate}
\item[(i)] $d_\tau(0)=0$;
\item[(ii)] $d_\tau([1])=1$;
\item[(iii)] $d_\tau(s+t)=d_\tau(s)+d_\tau(t)$ for all $s,t\in \Cu(A)$;
\item[(iv)] $d_\tau(s)\leq d_\tau(t)$ whenever $s\leq t$ in $\Cu(A)$;
\item[(v)] $d_\tau$ preserves suprema of increasing sequences.
\end{enumerate}
\end{prop}
Again with $A$ unital we use $\qt(A)$ to denote the set of all normalized quasitraces on $A$. Since we concern ourselves with nuclear $C^*$-algebras in this paper, by a theorem of Haagerup, we know that for an exact $C^*$-algebra, all quasitraces are traces and we write $T(A)$ for the tracial states on $A$. \\

\begin{defn}
    Suppose that $A$ is a $C^*$-algebra. Let $a,b\in A_+$. We say that $a$ is Cuntz-subequivalent to $b$, $a\precsim b$ whenever there exists a sequence $(x_i)_{i=1}^{\infty}$ in $A$ such that $a=\lim_i (x_ibx_i^* )$ in the norm topology.
\end{defn}

\begin{defn}
    Let $A$ be a $C^*$-algebra, and let $\mathcal{K}$ denote the compact operators on $\ell^2(\mathbb{N})$. We define an equivalence relation $\sim$ on $(A\otimes \mathcal{K})_+$ as follows: $a\sim b$ whenever both $a\precsim b$ and $b\precsim a$.
\end{defn}

\begin{defn}
    Let $A$ be a $C^*$-algebra, its Cuntz semigroup denoted $Cu(A)$ is defined as $(A\otimes \mathcal{K})_+/\sim$. The binary operation is defined by $\langle a\rangle+\langle b\rangle=\langle a\oplus b\rangle$.
\end{defn}

\begin{defn}\cite{GardellaPerera}
Let $(S,\leq)$ be an ordered set. We define an additional relation $\ll$ on $S$,
called (sequential) compact containment, as follows: for $s,t \in S$, we set
$s \ll t$ if whenever $(x_n)_{n\in\mathbb N}$ is an increasing sequence in $S$
with supremum $x$ satisfying $t \leq x$, then there exists $n \in \mathbb N$
such that $s \leq x_n$.

We say that $s \in S$ is compact if $s \ll s$.
\end{defn}

\begin{defn}\cite{GardellaPerera}\label{def:cuntz-semigroup-axioms}
Let $(S,\leq)$ be a positively ordered monoid. We say that $S$ is an
(abstract) Cuntz semigroup, or just a Cu-semigroup, if it satisfies the
following so-called axioms:

\begin{enumerate}
\item[(O1)] Every increasing sequence has a supremum.
\item[(O2)] For every $s \in S$, there is a sequence $(s_n)_{n\in\mathbb N}$ in $S$
with $s_n \ll s_{n+1}$ and $s = \sup_{n\in\mathbb N} s_n$.
\item[(O3)] If $s \ll t$ and $s' \ll t'$, then $s+s' \ll t+t'$.
\item[(O4)] If $(s_n)_{n\in\mathbb N}$ and $(t_n)_{n\in\mathbb N}$ are increasing
sequences, then
\[
\sup_{n\in\mathbb N}(s_n+t_n)
=
\sup_{n\in\mathbb N}s_n
+
\sup_{n\in\mathbb N}t_n.
\]
\end{enumerate}

Given Cu-semigroups $S$ and $T$, a Cu-morphism between them is a map
$f:S\to T$ preserving addition, neutral element, order $\leq$, suprema of
increasing sequences, and also the compact containment relation $\ll$. Maps
that preserve all the structure except possibly for $\ll$ are called generalized
Cu-morphisms.
\end{defn}

\begin{defn}
    Let $A$ be a $C^*$-algebra. We define the following spaces for clarity:
    \begin{align}
        \text{Aff}(T(A)):=&\{f:T(A)\to \mathbb{R}: f\text{ is continuous and affine}\}\\
        \text{Aff}(T(A))_+:=&\{f\in \text{Aff}(T(A)):f\geq 0\}\\
        \text{Aff}(T(A))_{++}:=&\{f\in \text{Aff}(T(A)):f> 0\}       
    \end{align}
 We also follow this notation for the set of lower-semicontinuous and affine functions $\text{LAff}\big(T(A)),[0,\infty]\big)$
\end{defn}

\begin{defn}
    Let $A$ be a unital $C^*$-algebra, suppose that $T(A)\neq \emptyset$. Let $\tau\in T(A)$ and $c\in M_{s}(A)_+$. Extend $\tau$ by 
    \begin{align}
        \tau^{(s)}:=\text{Tr}_{s}\otimes \tau
    \end{align}
    where $\text{Tr}_{s}$ is the unnormalized trace on $M_s(\mathbb{C})$. Further, define the amplified dimension function 
    \begin{align}
        \hat c(\tau):=d_{\tau^{(s)}}(c)
    \end{align}
\end{defn}

\begin{rmk}
    If in the definition above $A=M_n\otimes C(X)$ for $X$ a compact metric space, and $\tau=\tau_x$ an extreme trace, we get that $\hat c(\tau_x)=\frac{\text{rank}(c(x))}{n}$
\end{rmk}
\subsection{Hilbert Modules}\hspace{1in}\\

While the definition above for the Cuntz semigroup deals with the comparison and equivalence classes of positive elements after stabilization, it has been formulated in other ways, namely via the Cuntz equivalence classes of countably generated Hilbert $C^*$-modules over the algebra in question. For a concise treatment of Hilbert $C^*$-module theory, we refer the reader to \cite{Lance1995}.

We occasionally write $E_A$ to denote a generic (right) Hilbert module $E$ over a $C^*$-algebra $A$. 
\begin{defn}
    We denote by $\mathcal{L}(E_A)$ the set of all bounded, linear, adjointable operators $T:E_A\to E_A$ for a right Hilbert $A$-module $E_A$. We write $\mathcal{K}(E_A)$ to mean the norm closure of the ``finite rank" operators. 
\end{defn}

\begin{defn}
    Let $E$ and $F$ be two Hilbert $A$-modules such that $E\subseteq F$. We say $E$ is compactly contained in $F$, $E\subset\subset F$ if there exists a self-adjoint compact endomorphism $T$ of $F$ such that $T|_{E}=Id_E$.
\end{defn}

\begin{defn}
    Let $E$ and $F$ be two Hilbert $A$-modules. We write $E\precsim F$ if every compactly contained submodule $E'\subset\subset E$ is isomorphic to a compactly contained submodule of $F$. Cuntz equivalence is the symmetrization of this subequivalence.
\end{defn}

\begin{thm}\cite{GardellaPerera}
    Let $A$ be a $C^*$-algebra. The Cuntz semigroup of $A$, $Cu(A)$ is isomorphic to the semigroup generated by Cuntz equivalence classes of countably generated Hilbert $A$-modules endowed with direct sum as the binary operation. Moreover, the Hilbert $A$-module corresponding to a given positive element $a\in (A\otimes \mathcal{K})$ is $\overline{a(\ell^{2}(\mathbb{N})\otimes A)}$. 
\end{thm}

In the case where $A$ has stable rank equal to one, it is known that it is enough to consider embeddings as the subequivalence condition for the Hilbert module picture of $Cu(A)$. We remark that isomorphism of Hilbert $A$-modules implies Cuntz equivalence.\cite{ara2009ktheoryoperatoralgebrasclassification} \\

\subsection{AH algebras}\hspace{1in}\\

\begin{defn}\cite{toms2006flatdimensiongrowthcalgebras}
     A $C^*$-algebra is homogeneous if it is of the form $p(C(X)\otimes \mathcal{K})p$ for some $X$ compact and projection $p\in P(C(X)\otimes \mathcal{K})$ of rank $n\in \mathbb{N}$.
\end{defn}

\begin{defn}\cite{JACELON20153304}
    An approximately homogeneous (AH) $C^*$-algebra is an inductive limit of direct sums of homogeneous $C^*$-algebras.
\end{defn}

\begin{defn}\cite{JACELON20153304}
    A map $\phi_j:C(X_j)\to M_{R_j}(C(X_{j+1}))$ is called diagonal if it has the form 

    \begin{align}
        \phi_j(f)=\text{diag}(f\circ\pi_1,\dots, f\circ\pi_{R_j}):=\begin{bmatrix}
            f\circ\pi_1    & 0      & \cdots & 0      \\
            0      & f\circ\pi_2    & \cdots & 0      \\
            \vdots & \vdots & \ddots & \vdots \\
             0      & 0      & \cdots & f\circ\pi_{R_j}
        \end{bmatrix}
    \end{align}
    where each of the continuous maps $\pi_j:X_{j+1}\to X_{j}$ are called the eigenvalue maps of $\phi_j$.
\end{defn}

We are concerned with the situation where $X_{j+1}$ is formed by taking copies of $X_j$ and the eigenvalue maps are the coordinate projections and point evaluations.

\begin{defn}
    Let $(n_j)_{\in \mathbb{N}}$ be an increasing sequence of natural numbers, such that $n_{j}|n_{j+1}$ for all $j$. A unital, diagonal $*$-homomorphism 

    \begin{align}
        \phi_j:M_{m_{j}}\otimes C(X^{n_j})\to M_{m_{j+1}}\otimes C(X^{n_{j+1}}) 
    \end{align}
    is called a Villadsen map of the first type if each eigenvalue map is either a point evaluation or a coordinate projection. An inductive system $\mathcal{A}=\varinjlim(\mathcal A_j,\phi_j)$ with $\mathcal{A}_{j}=M_{m_j}\otimes C(X^{n_j})$ and $\phi_j$ Villadsen maps of the first type, is called a type I Villadsen algebra.
\end{defn}

\begin{defn}
    Let $A=M_{n}\otimes C(X)$. Suppose that $\text{diag}(f_1,\dots,f_n)= f\in A$ is a diagonal positive element. We say that $f$ is strictly diagonal if $\coz(f_n)\subseteq \coz(f_{n-1})\subseteq\dots\subseteq \coz(f_1).$ We call an element $g\in A_+$ strictly diagonalizable if $g\sim g'$ where $g'$ is strictly diagonal and the equivalence relation is Cuntz equivalence.   
\end{defn}

\begin{lem}\label{le2.17}\cite{Robert2013cuntz} 
    Let $I$ and $J$ be two $\sigma$-unital closed two-sided ideals of a $C^*$-algebra $A$ and $H $ a countably generated Hilbert $A$-module. The following isomorphism holds: 
    \begin{align}
        HI\oplus HJ\cong H(I+J)\oplus H(I\cap J)
    \end{align}
\end{lem}

\begin{lem}\label{le2.18}
    Let $a(x)\in M_{2}(C(X))^+=A$ be a diagonal element, where $X$ is compact and Hausdorff. It follows that $a(x)$ is strictly diagonalizable. 
\end{lem}
\begin{proof}
Suppose that $a(x)=\begin{bmatrix}
    f_1(x) & 0 \\ 0 & f_2(x)
\end{bmatrix}$ where $f_i: X\to \mathbb{R}_{\geq 0}$ and  $\text{Coz}(f_i)=U_i\subset X$.\\

The Hilbert $A$-module corresponding to $a(x)$ is $\overline{aA}\cong e_{11}M_2(C_0(U_1))\oplus e_{22}M_2(C_0(U_2))$. Noticing that $e_{11}J\cong e_{22}J$ via left multiplication of $e_{12}$. Thus, $\overline{aA}\cong e_{11}M_2(C_0(U_1))\oplus e_{11}M_2(C_0(U_2))=e_{11}AM_{2}(C_0(U_1))\oplus e_{11}AM_{2}(C_0(U_2))$ Invoking \ref{le2.17} with $H=e_{11}A$ and $I=M_2(C_0(U_1))$ and $J=M_2(C_0(U_2))$ we obtain an isomorphism of Hilbert $A$-modules
\begin{center}
    $\overline{aA}\cong \Big[e_{11}M_2(C_0(U_{1}\cup U_2))\oplus e_{11}M_2(C_0(U_1 \cap U_2))\Big]$
\end{center}

Let $g_1,g_2\in C(X)$ be functions with cozero sets  $U_1\cup U_2$ and $U_1 \cap U_2 $, respectively. Put $b(x):=\text{diag}(g_1,g_2)$. It then follows that the Hilbert $A$-module $\overline{bA}\cong \Big[e_{11}M_2(C_0(U_{1}\cup U_2))\oplus e_{11}M_2(C_0(U_1 \cap U_2))\Big]$. Thus $a(x)\sim_{cu} b(x)$ as desired.

\end{proof}

\begin{lem}\label{le2.19}
    Let $a(x)\in M_{n}(C(X))^+$ be a diagonal element, where $X$ is compact and Hausdorff. It follows that $a(x)$ is strictly diagonalizable. 
\end{lem}

\begin{proof}
    Let $a(x)=\text{diag}(f_1,\dots,f_n)$ where $\coz(f_i)=U_i\subset X$ for each $i=1,\dots ,n.$ We proceed inductively in the following way, \begin{align}
        a(x)=(f_1 \oplus f_2)\oplus \text{diag}(f_3,\dots,f_n)
    \end{align} From \ref{le2.18} we have that $f_1\oplus f_2\sim g_1 \oplus g_2$ with
    \begin{enumerate}
        \item  $\coz(g_1)=\coz(f_1)\cap \coz(f_2)$
        \item $\coz(g_2)=\coz(f_1)\cup \coz(f_2)$
    \end{enumerate}

Thus, $a(x)\sim \text{diag}(g_1,g_2,f_3,\dots,f_n)$. Now we write $\text{diag}(g_1,g_2,f_3,\dots,f_n)$ as $(g_1)\oplus (g_2\oplus f_3)\oplus \text{diag}(f_4,\dots,f_n)$. Apply \ref{le2.18} again to get $(g_2\oplus f_3)\sim (g_2'\oplus g_3)$ such that 
\begin{enumerate}
    \item  $\coz(g_2')=\coz(g_2)\cap \coz(f_3)$
        \item \begin{align}\coz(g_3)=&\coz(g_2)\cup \coz(f_3)\\ =& \coz(f_1)\cup \coz(f_2)\cup \coz(f_3)\end{align}
\end{enumerate}

Continue this process until we obtain an element $\text{diag}(g_1,g_2',\dots g_{n-1}',g_n)\sim a(x)$ with the property that $\coz(g_n)=\bigcup_{i=1}^{n} \coz(f_i)$.\\

Now we iterate again over the submatrix block $\text{diag}(g_1,\dots,g_{n-1}')$ to obtain an element $\text{diag}(g_1',g_2'',\dots,g_{n-2}'',g_{n-1}')\sim \text{diag}(g_1,\dots,g_{n-1}')$ this gives a new element $\text{diag}(g_1',g_2'',\dots,g_{n-2}'',g_{n-1}',g_{n})$ with the property that \begin{enumerate}
    \item $\coz(g_{n-1}')\subseteq\coz(g_n)$, and
$\coz(g_{n-1}')$ contains the cozero sets of all preceding
diagonal entries.
    \item $\coz(g_n)=\bigcup_{i=1}^{n} \coz(f_i)$
\end{enumerate}

We continue iterating this process until at the last stage, we have an element whose cozero sets are nested in increasing order. Abusing notation along the way we relabel and get that $a(x)\sim\text{diag}(g_1,\dots,g_n)$ such that $\coz(g_1)\subset \coz(g_2)\subset \dots \subset \coz(g_n)$. To get the desired element, apply a permutation unitary that reverses the order along the diagonal. 
\end{proof}

\subsection{Properties of the strictly diagonal elements. }

\noindent

If $a=\text{diag}(a_1,\dots,a_n)$ and $b=\text{diag}(b_1,\dots,b_n)$ are strictly diagonal elements, then it is easy to see that $a+b$ is strictly diagonal. The same is true of the product $ab$, $a^*$, and $\lambda a$ for all $\lambda\geq 0$ thus the set of strictly diagonal elements forms a $*$-closed, convex cone of $A=M_n\otimes C(X)$.

\begin{lem}\label{le2.20}
    Let $A=M_{n}\otimes C(X)$, where $X$ is a compact metric space. Suppose that $a,b\in A_+$ are strictly diagonal, then $a\precsim b$ if and only if $\text{rank} (a(x))\leq \text{rank}( b(x))$ for all $x\in X$.
\end{lem}
\begin{proof}
    Put $a=\text{diag}(a_1,\dots,a_n)$ and $b=\text{diag}(b_1,\dots,b_n)$. If $a\precsim b$, then it follows directly that $a(x)\precsim b(x) ,\forall x\in X$ which implies that $  \text{rank} (a(x))\leq \text{rank} (b(x))$ for all $x\in X$. Now assume that $\text{rank} (a(x))\leq \text{rank} (b(x))$ for all $x\in X$, then it follows that $\coz (a_i)\subset \coz(b_i)$ for each $i=1,\dots,n$. Thus, one has Cuntz subequivalence in each coordinate along the diagonal, which implies that $a\precsim b$.
\end{proof}

The important implication of \ref{le2.20} is that taken in conjunction with \ref{le2.18}, any positive, diagonal element can be represented by a strictly diagonal one, which then in turn implies that for any two diagonal elements $a$ and $b$ in $M_n\otimes C(X)$, one can find representatives whose Cuntz comparison can be carried out via their pointwise ranks. Thus, the Cuntz class recalls the underlying topology of the $C^*$-algebra.

Let $A=\varinjlim (A_i,\phi_i)$ be a type I Villadsen algebra. From \ref{le2.19} we see that the set of diagonal elements in each building block are Cuntz equivalent to some strictly diagonal element, and hence the diagonal maps do not disturb the Cuntz classes of strictly diagonal elements. \\

\begin{lem}\label{lem:affine_function_rep}\cite{BrownPereraToms2008} Let $g\in \mathrm{Aff}(T(A))$ satisfying $0\leq g<1$ where $A=M_{n}\otimes C(X)$ with $X$ a compact metric space and $n> 1$. Then there exists a strictly diagonal $b\in A_{+}$ such that $0\leq g - \hat b\leq \frac{1}{n}$. 
\end{lem}
\begin{proof}
     First, identify $T(A)\cong \text{Prob}(X)$. Set $q(x)=g(\tau_x)$ with $\tau_x:=tr_n\circ ev_x$. For each $0\leq i \leq n-1$ define an open set 
$$
A_{i}:=\{x\in X:q(x)> i/n\}.
$$
$\{x\in X:q(x)> \frac{i+1}{n}\}\subset \{x\in X:q(x)> \frac{i}{n}\}$ for each $i$. Since $X$ is a metric space, we can find a continuous map $f_{i}:X\to [0,1]$ such that $\text{Coz}(f_{i})=A_{i}$. \\

Put 
$$
B_{i}:=\{x\in X : \frac{i}{n}< q(x)\leq \frac{i+1}{n}\}=A_{i}\setminus (\cup_{j>i}A_{j}),
$$ for $0\leq i\leq n-2$ and $B_{n-1}:=A_{n-1}=\{x:\frac{n-1}{n}<q(x)<1\}$.
$$
a:=\bigoplus_{i=1}^{n-1}f_{i}e_{ii}=\mathrm{diag}(f_{1},\dots,f_{n-1},0)
$$
where $e_{ii}$ is the $i$-th standard matrix unit. 

$$
\hat a (\tau_x):= \frac{|1\leq j\leq n-1:x\in A_{j}|}{n}.
$$  
If $x\in (X\setminus A_{0})\cup B_{0}$, then $\hat a(x)=0$ and $0\leq g-\hat a \leq \frac{1}{n}$ for all such $x$, this follows from the definition of $B_{0}$. If $j\geq 1$ and $x\in B_{j}$, then $\hat a(\tau_x)=j/n$ and $\frac{j}{n}<q(x)\leq \frac{j+1}{n}$ and so $0\leq (g-\hat a)\leq  \frac{1}{n}$ for all such $x$. Since $\hat a$ is lower-semicontinuous, so is $\hat a-g$. A lower semi-continuous affine function on a Bauer simplex achieves its minimum on the extreme boundary, and the minimum is at least $\frac{-1}{n}$ by construction, thus $\hat a-g\geq \frac{-1}{n}$. By affineness, $\hat a-g\leq 0$ on any convex combination of extreme traces. By lower-semicontinuity $\hat a-g\leq 0$ on all of $T(A)$.

\end{proof}

\begin{cor}\label{cor:bounded_fn_rep}
   Let $A=M_{n}\otimes C(X)$, where $X$ is a compact metric space and let $g\in \text{Aff}(T(A))_+$ be bounded. Choose $k\in \mathbb{N}$ satisfying $||g||< k$. Then there exists a strictly diagonal element $b\in (M_{k}\otimes A)_+$ such that $0\leq g-\hat b\leq \frac{1}{n}$ on $T(A)$.
\end{cor}

\begin{proof}
Identify $T(M_{k}(A))\cong T(A)$ via $\tau\mapsto tr_{k}\otimes \tau$. Define $h(tr_{k}\otimes \tau)=\frac{1}{k}g(\tau)$, then $0\leq h <1$. Apply lemma \ref{lem:affine_function_rep} to $M_{k}(A)\cong M_{kn}(C(X))$. This provides the existence of $b\in M_{k}(A)_+$ such that 
\begin{align}
    0\leq& h(tr_{k}\otimes \tau)-d_{tr_k\otimes\tau}(b)\leq \frac{1}{kn}\\
    0\leq& \frac{1}{k}g(\tau)-d_{tr_k\otimes\tau}(b)\leq \frac{1}{kn}\\
    0\leq& g(\tau)-kd_{tr_k\otimes\tau}(b)\leq \frac{1}{n}\\
    0\leq& g(\tau)-d_{\tau^{(k)}}(b)\leq \frac{1}{n}
\end{align}

\end{proof}

\begin{defn}
\label{ass:positive-density}
We say that the inductive system has positive finite-stage density if the
following holds. For every strictly positive function
\[
H\in \Aff(T(A)),
\qquad
H(\tau)>0\quad \forall \tau\in T(A),
\]
for every \(\eps>0\), and for every initial stage \(N\in\mathbb N\), there
exist \(k\geq N\) and a function
\[
u_k\in \Aff(T(A_k))
\]
such that
\[
u_k(\sigma)>0
\qquad \forall \sigma\in T(A_k),
\]
and
\[
\left\|
\phi_{k,\infty}^{\bullet}(u_k)-H
\right\|<\eps.
\]
\end{defn}

\begin{rmk}
This is stronger than ordinary uniform density. The extra point is the
finite-stage positivity
\[
u_k>0\quad\text{on }T(A_k),
\]
not merely positivity after pullback to \(T(A)\). This is exactly the input
which lets us build a finite-stage gap rather than only a gap on the limit
trace space.
\end{rmk}

We now record a concrete condition under which
Assumption~\ref{ass:positive-density} holds for Villadsen systems of the first
type. This is the point at which the special form of the eigenvalue maps is
used.

Let
\[
A_i=M_{r_i}\otimes C(X_i)
\]
be a Villadsen building block. We identify traces on \(A_i\) with Borel
probability measures on \(X_i\), via
\[
\tau_\mu(a)=\int_{X_i}\operatorname{tr}_{r_i}(a(x))\,d\mu(x),
\qquad \mu\in \Prob(X_i),
\]
where \(\operatorname{tr}_{r_i}\) denotes the normalized matrix trace. Under
this identification,
\[
T(A_i)\cong \Prob(X_i).
\]

Suppose that the connecting map
\[
\phi_i:A_i\longrightarrow A_{i+1}
\]
has diagonal form
\[
\phi_i(f)=
\operatorname{diag}
\bigl(
 f\circ \lambda_{i,1},\ldots,f\circ \lambda_{i,N_i}
\bigr),
\]
where each eigenvalue map
\[
\lambda_{i,t}:X_{i+1}\to X_i
\]
is either a coordinate projection or a point evaluation. Let \(c_i\) denote the
number of coordinate projections among the eigenvalue maps at stage \(i\), and
let \(p_i\) denote the number of point evaluations. Thus
\[
N_i=c_i+p_i.
\]
Set
\[
\alpha_i:=\frac{c_i}{c_i+p_i}.
\]

\begin{defn}
\label{defn:negligible-point-tails}
For a type I Villadsen algebra we say that it has negligible point-evaluation tails if \(c_i>0\) for all sufficiently large \(i\), and 
\[
\lim_{k\to\infty}\prod_{\ell=k}^{\infty}\alpha_\ell=1.
\]
Equivalently,
\[
\lim_{k\to\infty}
\left(
1-\prod_{\ell=k}^{\infty}\frac{c_\ell}{c_\ell+p_\ell}
\right)=0.
\]
This holds, for instance, if
\[
\sum_{\ell=1}^{\infty}\frac{p_\ell}{c_\ell+p_\ell}<\infty.
\]
\end{defn}

\begin{lem}
\label{lem:tail-coordinate-traces}
Assume that ~\ref{defn:negligible-point-tails} holds. Fix a stage \(k\) and a
point \(x\in X_k\). Put
\[
\theta_k:=\prod_{\ell=k}^{\infty}\alpha_\ell.
\]
Then there exists a trace \(\tau_x\in T(A)\) and a probability measure
\(\rho_x\in \Prob(X_k)\) such that
\[
\phi_{k,\infty}^{\#}(\tau_x)
=
\theta_k\delta_x+(1-\theta_k)\rho_x.
\]
\end{lem}

\begin{proof}
For \(l>k\), consider the finite-stage map
\[
\phi_{k,l}:A_k\to A_l.
\]
The diagonal entries of \(\phi_{k,l}\) are indexed by finite paths of
eigenvalue maps from stage \(k\) to stage \(l\). A path contributes a
coordinate projection from \(X_l\) to \(X_k\) precisely when every eigenvalue
map along the path is a coordinate projection. The proportion of such purely
coordinate paths is
\[
\theta_{k,l}:=\prod_{\ell=k}^{l-1}\alpha_\ell.
\]
All remaining paths contain at least one point evaluation. Such paths
contribute maps \(X_l\to X_k\) which are independent of at least the coordinate
chosen at the first point-evaluation step; in particular, their contribution to
a trace on \(A_k\) is some probability measure on \(X_k\).

We claim that, for each \(l>k\), there is a point \(y_l\in X_l\) such that
restriction of the point-mass trace \(\delta_{y_l}\in \Prob(X_l)\) to
\(A_k\) has the form
\[
\phi_{k,l}^{\#}(\delta_{y_l})
=
\theta_{k,l}\delta_x+(1-\theta_{k,l})\rho_{x,l}
\]
for some \(\rho_{x,l}\in \Prob(X_k)\).

Indeed, the purely coordinate paths are coordinate projections from \(X_l\) to
\(X_k\). Choose the coordinates of \(y_l\) appearing in these purely coordinate
paths so that each of those coordinate projections takes the value \(x\). This
is possible because imposing the same value \(x\) on repeated coordinate blocks
is consistent. The remaining coordinates may be chosen arbitrarily. Then each
purely coordinate path contributes \(\delta_x\), while the non-pure paths
contribute some probability measure \(\rho_{x,l}\). This proves the claim.

Now consider the compact product
\[
\prod_{m=1}^{\infty} T(A_m).
\]
For each \(l>k\), the trace \(\delta_{y_l}\in T(A_l)\) determines a compatible
finite string
\[
\bigl(
\phi_{1,l}^{\#}(\delta_{y_l}),
\phi_{2,l}^{\#}(\delta_{y_l}),
\ldots,
\phi_{l-1,l}^{\#}(\delta_{y_l}),
\delta_{y_l}
\bigr).
\]
After filling in the remaining coordinates with a fixed tracial state from each building block. By compactness, passing to a subnet as \(l\to\infty\), these finite strings
converge coordinatewise to a compatible sequence
\[
(\tau_1,\tau_2,\tau_3,\ldots)
\in \varprojlim (T(A_m),\phi_m^{\#})=T(A).
\]
Let \(\tau_x\in T(A)\) be the trace corresponding to this compatible sequence.
The \(k\)-th coordinate of \(\tau_x\) is the limit of
\[
\phi_{k,l}^{\#}(\delta_{y_l})
=
\theta_{k,l}\delta_x+(1-\theta_{k,l})\rho_{x,l}.
\]
Since \(\Prob(X_k)\) is compact in the weak-* topology, after passing to a
further subnet we may assume
\[
\rho_{x,l}\to \rho_x
\]
for some \(\rho_x\in \Prob(X_k)\). Also,
\[
\theta_{k,l}\to \theta_k.
\]
Therefore
\[
\phi_{k,\infty}^{\#}(\tau_x)
=
\theta_k\delta_x+(1-\theta_k)\rho_x.
\]
This proves the lemma.
\end{proof}

Let $A = \varinjlim (A_i, \phi_i)$ be an inductive limit of homogeneous algebras. 

Each connecting map  
\begin{align}
\phi_i : A_i \to A_{i+1}
\end{align}

induces a continuous affine map  
\begin{align}
\phi_i^\# : T(A_{i+1}) \to T(A_i)
\end{align}

and a dual map  
\begin{align}
\phi_i^\bullet : \mathrm{Aff}(T(A_i)) \to \mathrm{Aff}(T(A_{i+1}))
\end{align}

given by  
\begin{align}
\phi_i^\bullet(f)(\tau)
= f\bigl(\phi_i^\#(\tau)\bigr),
\quad f \in \mathrm{Aff}(T(A_i)).
\end{align}

\begin{prop}
\label{prop:villadsen-positive-density}
Let
\[
A=\varinjlim(A_i,\phi_i)
\]
be a Villadsen algebra of the first type with negligible point-evaluations. Then the inductive system has
positive finite-stage density in the sense of
definition~\ref{ass:positive-density}.
\end{prop}

\begin{proof}
Let
\[
H\in \Aff(T(A))
\]
be strictly positive. Since \(T(A)\) is compact and \(H\) is continuous, there
is a number \(m>0\) such that
\[
H(\tau)\geq m
\qquad \forall \tau\in T(A).
\]
Fix \(\eps>0\) and an initial stage \(N\).

By ordinary uniform density of finite-stage affine functions in
\(\Aff(T(A))\), there are a stage \(i\geq N\) and a function
\[
h_i\in \Aff(T(A_i))
\]
such that
\[
\left\|
\phi_{i,\infty}^{\bullet}(h_i)-H
\right\|
<
\min\left\{\frac{\eps}{2},\frac{m}{4}\right\}.
\]
Let
\[
B:=\|h_i\|+1.
\]

Choose \(k\geq i\) so large that
\[
1-\theta_k<\frac{m}{4B},
\qquad
\theta_k:=\prod_{\ell=k}^{\infty}\alpha_\ell.
\]
This is possible by Assumption~\ref{defn:negligible-point-tails}. Define
\[
u_k:=\phi_{i,k}^{\bullet}(h_i)
\in \Aff(T(A_k)).
\]
Then
\[
\phi_{k,\infty}^{\bullet}(u_k)
=
\phi_{i,\infty}^{\bullet}(h_i),
\]
and hence
\[
\left\|
\phi_{k,\infty}^{\bullet}(u_k)-H
\right\|<\eps.
\]
It remains to prove that
\[
u_k>0\quad\text{on }T(A_k).
\]
Since \(T(A_k)=\Prob(X_k)\), and \(u_k\) is affine and continuous, it is enough
to prove positivity on extreme traces, i.e. on point masses \(\delta_x\) for
\(x\in X_k\).

Fix \(x\in X_k\). By Lemma~\ref{lem:tail-coordinate-traces}, there exist
\(\tau_x\in T(A)\) and \(\rho_x\in \Prob(X_k)\) such that
\[
\phi_{k,\infty}^{\#}(\tau_x)
=
\theta_k\delta_x+(1-\theta_k)\rho_x.
\]
Now
\[
\phi_{k,\infty}^{\bullet}(u_k)(\tau_x)
=
\phi_{i,\infty}^{\bullet}(h_i)(\tau_x).
\]
By the choice of \(h_i\),
\[
\phi_{i,\infty}^{\bullet}(h_i)(\tau_x)
>
H(\tau_x)-\frac{m}{4}
\geq
\frac{3m}{4}.
\]
On the other hand,
\[
\begin{aligned}
\phi_{k,\infty}^{\bullet}(u_k)(\tau_x)
&=
u_k\bigl(\phi_{k,\infty}^{\#}(\tau_x)\bigr)\\
&=u_k\bigl(\theta_k\delta_x+(1-\theta_k)\rho_x\bigr)\\
&=\theta_k u_k(\delta_x)+(1-\theta_k)u_k(\rho_x).
\end{aligned}
\]
Therefore
\[
\theta_k u_k(\delta_x)+(1-\theta_k)u_k(\rho_x)>
\frac{3m}{4}.
\]
Since
\[
\|u_k\|=\|\phi_{i,k}^{\bullet}(h_i)\|\leq \|h_i\|<B,
\]
we have
\[
u_k(\rho_x)\leq B.
\]
Hence
\[
\begin{aligned}
\theta_k u_k(\delta_x)
&>
\frac{3m}{4}-(1-\theta_k)|u_k(\rho_x)|\\
&\geq
\frac{3m}{4}-(1-\theta_k)B\\
&>
\frac{3m}{4}-\frac{m}{4}
=
\frac{m}{2}.
\end{aligned}
\]
Since \(\theta_k>0\), it follows that
\[
u_k(\delta_x)>0.
\]
This holds for every \(x\in X_k\). Therefore \(u_k\) is strictly positive on
all extreme traces of \(T(A_k)\), and hence, by affinity, on all of \(T(A_k)\).
Thus
\[
u_k(\sigma)>0
\qquad \forall \sigma\in T(A_k).
\]
We have produced \(k\geq N\) and \(u_k\in \Aff(T(A_k))\) such that
\[
u_k>0\quad\text{on }T(A_k),
\]
and
\[
\left\|
\phi_{k,\infty}^{\bullet}(u_k)-H
\right\|<\eps.
\]
This proves positive finite-stage density.
\end{proof}

\begin{cor}
\label{cor:specified-positive}
Let $s\in\mathbb N$ and let
\[
 w\in\Aff(T(A_s)).
\]
Suppose that
\[
 \phi_{s,\infty}^{\bullet}(w)>0
 \qquad\text{on }T(A).
\]
Then there exists $l\geq s$ such that
\[
 \phi_{s,l}^{\bullet}(w)>0
 \qquad\text{on }T(A_l).
\]
\end{cor}

\begin{proof}
Set
\[
 H=\phi_{s,\infty}^{\bullet}(w).
\]
Since $H$ is continuous and strictly positive on the compact space $T(A)$,
there exists $m>0$ such that
\[
 H(\tau)\geq m
 \qquad\text{for every }\tau\in T(A).
\]
Set
\[
 B=\|w\|+1.
\]
Choose $l\geq s$ so large that
\[
 (1-\theta_l)B<\frac{m}{4}.
\]
Put
\[
 W=\phi_{s,l}^{\bullet}(w)\in\Aff(T(A_l)).
\]

Fix $x\in X_l$.  By the tail-coordinate trace lemma, there exist
$\tau_x\in T(A)$ and $\rho_x\in\Prob(X_l)$ such that
\[
 \phi_{l,\infty}^{\#}(\tau_x)
 =\theta_l\delta_x+(1-\theta_l)\rho_x.
\]
Since
\[
 \phi_{l,\infty}^{\bullet}(W)
 =\phi_{s,\infty}^{\bullet}(w)
 =H,
\]
we have
\[
 \begin{aligned}
 m
 &\leq H(\tau_x)\\
 &=W\bigl(\phi_{l,\infty}^{\#}(\tau_x)\bigr)\\
 &=\theta_lW(\delta_x)+(1-\theta_l)W(\rho_x).
 \end{aligned}
\]
Moreover,
\[
 |W(\rho_x)|\leq\|W\|\leq\|w\|<B.
\]
Therefore
\[
 \begin{aligned}
 \theta_lW(\delta_x)
 &\geq m-(1-\theta_l)W(\rho_x)\\
 &\geq m-(1-\theta_l)B\\
 &>\frac{3m}{4}.
 \end{aligned}
\]
Thus
\[
 W(\delta_x)>0
 \qquad\text{for every }x\in X_l.
\]
Since $T(A_l)\cong\Prob(X_l)$ and $W$ is affine, positivity on all point
masses implies
\[
 W>0
 \qquad\text{on }T(A_l).
\]
Hence
\[
 \phi_{s,l}^{\bullet}(w)>0
 \qquad\text{on }T(A_l).
\]

This is exactly the proof of Proposition~\ref{prop:villadsen-positive-density} with the
initial finite-stage approximation chosen to be $w$ itself and with
approximation error zero.
\end{proof}

\begin{lem}
\label{lem:recursive-selection}
Let $A=\varinjlim(M_{r_{i}}\otimes C(X_{i}),\phi_{i})$ be a type I Villadsen algebra with negligible point evaluations and with $\lim_i r_i=\infty$. Let $0\neq f$ be a strictly positive lower semicontinuous affine function on
$T(A)$.  Choose
\[
 0<h_1<h_2<\cdots,
 \qquad
 h_n\in\Aff(T(A)),
 \qquad
 \sup_n h_n=f.
\]
Then there exist increasing stages $j_n$, functions
$g_n\in\Aff(T(A_{j_n}))_+$, and numbers $\eta_n\downarrow0$ such that,
writing $G_n=\phi_{j_n,\infty}^\bullet(g_n)$,
\begin{align}
 &\frac1{r_{j_n}}<\eta_n,\\
 &\|G_n-h_n\|\longrightarrow0,\\
 &G_n+\eta_n<h_n,\\
 &\phi_{j_n,j_{n+1}}^\bullet(g_n)+\eta_n+\eta_{n+1}<g_{n+1}
   \quad\text{on }T(A_{j_{n+1}}).
\end{align}
\end{lem}

\begin{proof}
Set
\[
 \Delta_n=\min_{\tau\in T(A)}(h_{n+1}-h_n)(\tau)>0.
\]
Choose $\eps_n\downarrow0$ with
$4\eps_{n+1}<\Delta_n$ and
$4\eps_1<\min h_1$.  Put
\[
 \eta_n=\frac18\min\{\eps_n,\eps_{n+1}\}.
\]
At each step choose $j_n$ sufficiently large that
$1/r_{j_n}<\eta_n$.

Apply Proposition~\ref{prop:villadsen-positive-density} to $h_1-2\eps_1$ with error
$\eps_1/2$ to obtain $j_1$ and $g_1>0$ satisfying
\[
 h_1-\frac52\eps_1<G_1<h_1-\frac32\eps_1.
\]
Suppose $j_n,g_n$ have been chosen with
\[
 h_n-3\eps_n<G_n<h_n-\eps_n.
\]
Then
\[
 H_{n+1}:=h_{n+1}-G_n-2\eps_{n+1}
\]
is strictly positive because
$h_{n+1}-G_n>\Delta_n>4\eps_{n+1}$.  Apply positive finite-stage density to
$H_{n+1}$ with error $\eps_{n+1}/2$, at a stage $j_{n+1}>j_n$, to obtain
$u_{n+1}>0$.  Define
\[
 g_{n+1}
 =\phi_{j_n,j_{n+1}}^\bullet(g_n)
  +\eta_n+\eta_{n+1}+u_{n+1}.
\]
The strict finite-stage gap is immediate.  Pulling to $T(A)$ and using the
error bounds gives
\[
 h_{n+1}-3\eps_{n+1}<G_{n+1}<h_{n+1}-\eps_{n+1}.
\]
This proves the induction.
\end{proof}

\begin{lem}
\label{lem:compatibility_along_sequence}
Suppose that
\[
A=\varinjlim(A_i,\phi_i),
\qquad
A_i=M_{r_i}\otimes C(X_i),
\]
where each \(X_i\) is compact metrizable and the connecting maps are diagonal with negligible point evaluations.
Let \(i<j\). Let
\[
g_i\in \Aff(T(A_i))_+,
\qquad
 g_j\in \Aff(T(A_j))_+,
\]
be both bounded and let
\[
a_i\in M_{k(g_i)}(A_i)_+,
\qquad
 a_j\in M_{k(g_j)}(A_j)_+
\]
be the strictly diagonal positive elements corresponding to $g_i$ and $g_j$ as in lemma \ref{cor:bounded_fn_rep}. Suppose that there are constants
\(\eta_i,\eta_j>0\) such that
\[
0\leq g_i-\widehat{a_i}<\eta_i
\quad\text{on }T(A_i),
\]
and
\[
0\leq g_j -\widehat{a_j}<\eta_j
\quad\text{on }T(A_j).
\]
Assume also that, after mapping \(g_i\) forward to \(T(A_j)\), one has
\[
\phi_{i,j}^{\bullet}(g_i)+\eta_i+\eta_j<g_j
\quad\text{on }T(A_j).
\]
Then
\[
\phi_{i,j}^{(k(g_i))}(a_i)\precsim a_j
\]
in \(A_j\otimes \mathcal{K}\). Consequently,
\[
\phi_{i,\infty}^{(k(g_i))}(a_i)\precsim \phi_{j,\infty}^{(k(g_j))}(a_j)
\]
in \(A\otimes \mathcal K\), $\phi^{(\ell)}$ denotes the amplification of a map $\phi$ up to $M_{\ell}$, namely $\phi_{i,j}^{(\ell)}:=\text{id}_{M_{\ell}}\otimes \phi_{i,j}$. Specifically,

\[
\phi_{i,\infty}\otimes \text{id}_{\mathcal{K}}(a_i)\precsim \phi_{j,\infty}^{(k(g_j))}(a_j)
\]
\end{lem}

\begin{proof}
Let $k_i:=k(g_i)$ and $k_j:=k(g_j)$. We prove the finite-stage comparison
\begin{align}
\phi_{i,j}^{(k_{i})}(a_i)\precsim a_j
\end{align}
inside \(A_j\otimes \mathcal K\). By padding both elements $\phi_{i,j}^{(k_{i})}(a_i)$ and $(a_j)$ with enough zeros so that they sit inside the same matrix algebra $M_N\otimes A_j$, we can then apply the limit map \(\phi_{j,\infty}^{N}\) so that it gives
\begin{align}
\phi_{j,\infty}^{N}(\phi_{i,j}^{(k_{i})}(a_i)\oplus 0)
\precsim
\phi_{j,\infty}^{N}(a_j\oplus 0),
\end{align}
because \(*\)-homomorphisms preserve Cuntz subequivalence.

Let \(\tau\in T(A_j)\). By definition of the dual affine map,
\begin{align}
\phi_{i,j}^{\bullet}(h)(\tau)=
h(\tau\circ \phi_{i,j})
\qquad
(h\in \Aff(T(A_i))).
\end{align}
We extend $\tau\in T(A_j)$ to $\tau^{k_i}:=\text{Tr}_{k_i}\otimes \tau$, where $\text{Tr}_k$ is the unnormalized trace on $M_k$.
Therefore, for the dimension function induced by \(a_i\), we have
\begin{align}
\widehat{\phi_{i,j}^{(k_{i})}(a_i)}(\tau)
&=
d_{\tau^{(k_i)}}(\phi_{i,j}^{(k_{i})}(a_i))\\
&=
d_{\tau^{(k_{i})}\circ \phi_{i,j}^{(k_{i})}}(a_i)\\
&=
\widehat{a_i}(\tau\circ \phi_{i,j})\\
&=
\phi_{i,j}^{\bullet}(\widehat{a_i})(\tau).
\end{align}
Using the hypothesis
\begin{align}
0\leq g_i-\widehat{a_i}<\eta_i
\quad\text{on }T(A_i),
\end{align}
we get
\begin{align}
\widehat{a_i}\leq g_i.
\end{align}
Pulling this inequality forward by \(\phi_{i,j}^{\bullet}\), we obtain
\begin{align}
\phi_{i,j}^{\bullet}(\widehat{a_i})
\leq
\phi_{i,j}^{\bullet}(g_i)
\quad\text{on } T(A_{j}).
\end{align}
Hence
\begin{align}
\widehat{\phi_{i,j}^{(k_i)}(a_i)}
\leq
\phi_{i,j}^{\bullet}(g_i).
\end{align}
By our hypothesis,
\begin{align}
\phi_{i,j}^{\bullet}(g_i)+\eta_i+\eta_j<g_j.
\end{align}
Subtracting \(\eta_j\) from both sides gives
\begin{align}
\phi_{i,j}^{\bullet}(g_i)+\eta_i<g_j-\eta_j.
\end{align}
Consequently,
\begin{align}
\widehat{\phi_{i,j}^{(k_i)}(a_i)}<g_j-\eta_j.
\end{align}
On the other hand, the second approximation inequality says
\begin{align}
0\leq g_j-\widehat{a_j}<\eta_j.
\end{align}
In particular,
\begin{align}
g_j-\eta_j<\widehat{a_j}\leq g_j.
\end{align}
Combining the previous estimates gives
\begin{align}
\widehat{\phi_{i,j}^{(k_i)}(a_i)}
<
g_j-\eta_j
<
\widehat{a_j}.
\end{align}
Thus
\begin{align}
\widehat{\phi_{i,j}^{(k_i)}(a_i)}(\tau)
<
\widehat{a_j}(\tau)
\qquad
\forall \tau\in T(A_j).
\end{align}
In particular,
\begin{align}
\widehat{\phi_{i,j}^{(k_i)}(a_i)}(\tau)
\leq
\widehat{a_j}(\tau)
\qquad
\forall \tau\in T(A_j).
\end{align}

We now convert this trace inequality into a pointwise rank inequality. Write
\begin{align}
A_j=M_{r_j}\otimes C(X_j).
\end{align}
For each \(x\in X_j\), let \(\tau_x\in T(A_j)\) be the trace obtained by
first evaluating at \(x\) and then applying the normalized matrix trace on
\(M_{r_j}\). Explicitly,
\begin{align}
\tau_x(b)=\frac{1}{r_j}\operatorname{Tr}(b(x)),
\qquad
b\in A_j.
\end{align}
For any positive element \(b\in M_{s}(A_j)_+\), the associated dimension function at
\(\tau_x^{(s)}\) is
\begin{align}
d_{\tau_x^{(s)}}(b)
&=
\lim_{n\to\infty}\tau_x^{(s)}(b^{1/n})\\
&=
\lim_{n\to\infty}\frac{1}{r_j}\operatorname{Tr}(b(x)^{1/n}).
\end{align}
If the nonzero eigenvalues of \(b(x)\) are \(\lambda_1,\ldots,\lambda_m>0\),
then the nonzero eigenvalues of \(b(x)^{1/n}\) are
\begin{align}
\lambda_1^{1/n},\ldots,\lambda_m^{1/n}.
\end{align}
For each \(1\leq \ell\leq m\),
\begin{align}
\lim_{n\to\infty}\lambda_\ell^{1/n}=1,
\end{align}
while zero eigenvalues remain zero. Therefore
\begin{align}
\lim_{n\to\infty}\operatorname{Tr}(b(x)^{1/n})=
\operatorname{rank}(b(x)).
\end{align}
Thus
\begin{align}
d_{\tau_x}(b)=\frac{\operatorname{rank}(b(x))}{r_j}.
\end{align}
Applying this to \(b=\phi_{i,j}^{(k_i)}(a_i)\) and to \(b=a_j\), the inequality
\begin{align}
\widehat{\phi_{i,j}^{(k_i)}(a_i)}(\tau_x)
\leq
\widehat{a_j}(\tau_x)
\end{align}
becomes
\begin{align}
\frac{\operatorname{rank}(\phi_{i,j}^{(k(g_{i}))}(a_i)(x))}{r_j}
\leq
\frac{\operatorname{rank}(a_j(x))}{r_j}.
\end{align}
Multiplying by \(r_j\), we get
\begin{align}
\operatorname{rank}(\phi_{i,j}^{(k_i)}(a_i)(x))
\leq
\operatorname{rank}(a_j(x))
\qquad
\forall x\in X_j.
\end{align}

The element \(a_j\) is strictly diagonal by hypothesis. The element
\(\phi_{i,j}^{(k_i)}(a_i)\) is diagonal because \(\phi_{i,j}\) is diagonal, but its
entries may not be ordered so as to be strictly diagonal. By
Lemma~\ref{le2.19}, there exists a strictly diagonal positive element
\begin{align}
c\in (A_j\otimes M_{k_i})_+
\end{align}
such that
\begin{align}
c\sim_{\Cu}\phi_{i,j}^{(k(g_{i}))}(a_i).
\end{align}
Cuntz equivalent positive elements have the same dimension functions; hence
\begin{align}
\widehat c=\widehat{\phi_{i,j}^{(k(g_{i}))}(a_i)}.
\end{align}
Evaluating again at \(\tau_x\), we obtain
\begin{align}
\frac{\operatorname{rank}(c(x))}{r_j}
&=
\widehat c(\tau_x)\\
&=
\widehat{\phi_{i,j}^{(k_i)}(a_i)}(\tau_x)\\
&\leq
\widehat{a_j}(\tau_x)\\
&=
\frac{\operatorname{rank}(a_j(x))}{r_j}.
\end{align}
Therefore
\begin{align}
\operatorname{rank}(c(x))
\leq
\operatorname{rank}(a_j(x))
\qquad
\forall x\in X_j.
\end{align}
Now both \(c\) and \(a_j\) are strictly diagonal. Now we pad the elements $c,\phi_{i,j}^{(k_i)}(a_i)$ and $a_j$ so that Lemma~\ref{le2.20}
applies and gives
\begin{align}
c\precsim a_j.
\end{align}
Since
\begin{align}
\phi_{i,j}^{(k_i)}(a_i)\sim_{\Cu} c,
\end{align}
we conclude that
\begin{align}
\phi_{i,j}^{(k_i)}(a_i)\precsim a_j.
\end{align}
This proves the desired finite-stage comparison. Applying \(\phi_{j,\infty}\otimes \text{id}_{\mathcal K}\)
gives
\begin{align}
\phi_{i,\infty}^{(k_i)}(a_i)
&=(
\phi_{j,\infty}\otimes \text{id}_{\mathcal K})(\phi_{i,j}^{(k(g_{i}))}(a_i))\\
&\precsim
\phi_{j,\infty}^{(k_j)}(a_j).
\end{align}

\end{proof}

\begin{defn}
\label{def:admissible-realization}
Let $f\in\LAff(T(A))_{++}$. With the notation above, an
\emph{admissible realization} of $f$ consists of strictly increasing
stages $(j_n)$, functions
$f_n\in\Aff(T(A))_{++}$ and
$u_n\in\Aff(T(A_{j_n}))_{++}$, numbers $\eta_n\downarrow0$,
and strictly diagonal positive elements
$a_n\in M_{\ell_n}(A_{j_n})_+$, for some $\ell_n\in\mathbb N$,
satisfying
\begin{align}
    &0<f_n<f_{n+1},
    \qquad \sup_n f_n=f, \label{eq:admissible-approximants}\\
    &\|U_n-f_n\|_\infty\longrightarrow0,
    \qquad U_n+\eta_n<f_n
    \quad\text{on }T(A), \label{eq:admissible-limit-gap}\\
    &0\leq u_n-\widehat{a_n}<\eta_n
    \quad\text{on }T(A_{j_n}), \label{eq:admissible-rank-error}\\
    &\phi_{j_n,j_{n+1}}^\bullet(u_n)
      +\eta_n+\eta_{n+1}<u_{n+1}
    \quad\text{on }T(A_{j_{n+1}}),
    \label{eq:admissible-stage-gap}
\end{align}
where $U_n=\phi_{j_n,\infty}^\bullet(u_n)$ and
$\widehat{a_n}(\tau)
=d_{\operatorname{Tr}_{\ell_n}\otimes\tau}(a_n)$,
with $\operatorname{Tr}_{\ell_n}$ the unnormalized matrix trace.

The Cuntz class associated with this realization is
\[
    x_f:=\sup_n x_n,
    \qquad
    x_n=
    \bigl[(\phi_{j_n,\infty}\otimes\mathrm{id}_{\K})(a_n)\bigr],
\]
where each $a_n$ is viewed in $A_{j_n}\otimes\K$.
The sequence $(x_n)$ is increasing by
Lemma~\ref{lem:compatibility_along_sequence}, so this supremum
exists in $\Cu(A)$.
\end{defn}

For later reference, the notation used for two realizations is summarized as
follows:
\[
\begin{array}{c|c|c}
 & \text{realization of }f & \text{realization of }g\\ \hline
\text{continuous approximants}
 & f_n\uparrow f & g_m\uparrow g\\
\text{finite-stage affine functions}
 & u_n\in\Aff(T(A_{j_n})) & v_m\in\Aff(T(A_{k_m}))\\
\text{pullbacks to }T(A)
 & U_n=\phi_{j_n,\infty}^{\bullet}(u_n)
 & V_m=\phi_{k_m,\infty}^{\bullet}(v_m)\\
\text{error tolerances}
 & \eta_n & \zeta_m\\
\text{positive elements}
 & a_n\in A_{j_n}\otimes\K
 & b_m\in A_{k_m}\otimes\K\\
\text{Cuntz classes}
 & x_n=[\phi_{j_n,\infty}(a_n)]
 & y_m=[\phi_{k_m,\infty}(b_m)]
\end{array}
\]
Thus $u_n$ and $v_m$ live at finite stages, while $U_n$ and $V_m$ are the
same functions viewed on the limit trace space. 

\begin{thm}
\label{thm:interleaving}
Let $f,g\in\LAff_{++}(T(A))$.  Let
\[
 x_f=\sup_nx_n
 \qquad\text{and}\qquad
 x_g=\sup_my_m
\]
be arbitrary admissible realizations, with notation as in the table above.  If
\[
 f\leq g,
\]
then
\[
 x_f\leq x_g.
\]
\end{thm}

\begin{proof}
Fix $n$.  We will show that the $n$th term $x_n$ of the realization of $f$ is
below some term $y_m$ of the realization of $g$.

The admissibility inequality for the first realization is
\[
 U_n+\eta_n<f_n\leq f\leq g.
\]
Set
\[
 q_n:=U_n+\eta_n.
\]
Then $g-q_n$ is lower semicontinuous and strictly positive on the compact
space $T(A)$.  Hence there exists $\delta_n>0$ such that
\[
 q_n+\delta_n\leq g
 \qquad\text{on }T(A).
\]

The functions $V_m$ are increasing.  Indeed, the recursive finite-stage gap
for the second realization gives
\[
 V_m+\zeta_m+\zeta_{m+1}<V_{m+1}.
\]
Furthermore,
\[
 \sup_mV_m=g,
\]
because
\[
 \|V_m-g_m\|\longrightarrow0
 \qquad\text{and}\qquad
 g_m\uparrow g.
\]

For each $m$, define
\[
 \mathcal U_m
 =\left\{
 \tau\in T(A):
 V_m(\tau)>q_n(\tau)+\frac{3\delta_n}{4}
 \right\}.
\]
The sets $\mathcal U_m$ are open and increasing.  They cover $T(A)$ because
\[
 q_n+\frac{3\delta_n}{4}<g=\sup_mV_m.
\]
By compactness, there exists $m_0$ such that
\[
 \mathcal U_{m_0}=T(A).
\]
Choose $m\geq m_0$ so large that
\[
 \zeta_m<\frac{\delta_n}{4}.
\]
Then
\[
 \begin{aligned}
 V_m-U_n-\eta_n-\zeta_m
 &=V_m-q_n-\zeta_m\\
 &>\frac{\delta_n}{2}
 \end{aligned}
\]
on $T(A)$.

Choose a common stage
\[
 s\geq j_n,k_m
\]
and define the finite-stage difference
\[
 w
 :=\phi_{k_m,s}^{\bullet}(v_m)
 -\phi_{j_n,s}^{\bullet}(u_n)
 -\eta_n-\zeta_m
 \in\Aff(T(A_s)).
\]
Its pullback to the limit trace space is exactly the positive cross-gap above:
\[
 \begin{aligned}
 \phi_{s,\infty}^{\bullet}(w)
 &=V_m-U_n-\eta_n-\zeta_m\\
 &>\frac{\delta_n}{2}.
 \end{aligned}
\]

Apply Corollary~\ref{cor:specified-positive}, the specified finite-stage
consequence of Proposition~\ref{prop:villadsen-positive-density}.  There exists
$l\geq s$ such that
\[
 \phi_{s,l}^{\bullet}(w)>0
 \qquad\text{on }T(A_l).
\]
Expanding the definition of $w$, this says
\[
 \phi_{j_n,l}^{\bullet}(u_n)+\eta_n+\zeta_m
 <\phi_{k_m,l}^{\bullet}(v_m)
 \quad\text{on }T(A_l).
\]
Lemma~\ref{lem:compatibility_along_sequence} therefore gives (after replacing the elements with Cuntz-equivalent strictly diagonal elements)
\[
 \phi_{j_n,l}(a_n)\precsim\phi_{k_m,l}(b_m).
\]
Passing to the limit yields
\[
 x_n\leq y_m\leq x_g.
\]
Since this holds for every $n$,
\[
 x_f=\sup_nx_n\leq x_g.
\]
\end{proof}

\subsection{Spectral profiles and unitary orbit metrics}

The investigation of positive elements also requires the following notions from \cite{JACELON20153304} concerning the distance between unitary orbits and the spectral profile of a positive element.\\

The Serre--Swan theorem \cite{Swan1962} asserts that the functor of
continuous sections gives an equivalence between the category
of finite-rank complex vector bundles over a compact Hausdorff
space $X$ and the category of finitely generated projective
$C(X)$-modules. Consequently, every such vector bundle $\xi$
is isomorphic to the image bundle of a projection
$p_\xi\in M_N(C(X))$ for some $N\in\mathbb N$.
Moreover, two projections represent isomorphic vector bundles
if and only if they are Murray--von Neumann equivalent after
embedding them into a common matrix algebra by direct summing each with an appropriately sized zero element.

\begin{defn}
    Let $A$ be a $C^*$-algebra. We define the distance between the unitary orbits of two positive elements $a,b\in A$ as \newline $d_U(a,b)=\inf \Big\{||uau^*-b||:u\in \mathcal{U}(\tilde A)\Big\}$
\end{defn}

\begin{defn}\cite{JACELON20153304}
    Let $A$ be a unital, stably finite, exact $C^*$-algebra. For two positive contractions $a,b\in A$, we define $d_P(a,b)$ to be the infimum over all $r>0$ such that $\mu_{\tau,a}(U_r)\geq \mu_{\tau,b}(U)$ and $\mu_{\tau,b}(U_r)\geq \mu_{\tau,a}(U)$ for every open subset $U\subset (0,1]$ and every trace $\tau\in T(A)$, $\mu_{\tau,a}$ and $\mu_{\tau,b}$ are the Borel measures induced by $\tau $ on $C^*(a)$ and $C^*(b)$ respectively, and $U_r:=\{t:\text{dist}(t,U)<r\}$. 
\end{defn}

\begin{defn}\label{def2.26}
    Let $A$ be a unital $C^*$-algebra. For every positive contraction $a\in A$, we define the spectral profile $f_a:T(A)\to \text{Prob}([0,1])^+$ by $f_a(\tau)=\mu_{\tau,a}$ where $\mu_{\tau,a}$ is the Radon measure induced by $\tau$ on $\sigma(a)$ by the Riesz-Markov-Kakutani representation theorem. 
\end{defn}

Note that if $f_a=f_b$ then $d_P(a,b)=0$.

\begin{lem}\label{le2.27}
    Let $a,b\in M_{n}(C(X))^+$ where $X$ is compact and Hausdorff and $a$ and $b$ are contractions. 
    Then $\rk((a(x)-\epsilon)_+)=\rk((b(x)-\epsilon)_+)$ for all $x\in X$ and all $\epsilon\geq 0$ if and only if  $f_a (\tau)=f_b(\tau)$ for all $\tau\in \T(M_n(C(X)))$.
\end{lem}
\begin{proof}
We claim that $\rk(a(x))=\rk(b(x))$ for all $x\in X$ implies that $d_{\tau}(a)=d_{\tau}(b)$ for all $\tau\in T(M_{n}(C(X))$. 
By identifying $M_{n}(C(X))$ with $C(X,M_n)$ we get that for any tracial state $\tau\in T(M_{n}(C(X)))$ there exists a Radon measure $\mu_{\tau}$ on $X$ such that \cite{Mortari2009}  

\begin{align}
\tau(a)=\int_{X}tr_n(a(x))\mu_{\tau}(x)\label{2.1}
\end{align}

So when we consider the rank function of $a\in M_{n}(C(X))^+$ and equation \eqref{2.1} we get a useful connection between traces on the amplification and the canonical trace $tr_n$ on $M_{n}(\mathbb{C})$. 
\begin{align}
d_{\tau}(a)=\lim_{m\to\infty}\tau(a^{1/m})=\lim_{m\to\infty}\int_{X} tr_{n}(a(x)^{1/m})\mu_{\tau}(x)\end{align}

We can now apply the dominated convergence theorem along with our hypothesis that $\rk(a(x))=\rk(b(x))$ and obtain: 

\begin{align}
d_{\tau}(a)=&\lim_{m\to\infty}\tau(a^{1/m})\\
=&\lim_{m\to\infty}\int_{X} tr_{n}(a(x)^{1/m})\mu_{\tau}(x)\\
=&\int_{X} \lim_{m\to\infty}tr_{n}(a(x)^{1/m})\mu_{\tau}(x)\\
=& \int_{X} \lim_{m\to\infty}tr_{n}(b(x)^{1/m})\mu_{\tau}(x)\\
=&\lim_{m\to\infty}\int_{X} tr_{n}(b(x)^{1/m})\mu_{\tau}(x)\\
=&d_{\tau}(b)\\
\end{align}

Now we can show that the spectral profiles of $a$ and $b$ are equal. That is, for any trace $\tau\in T(M_{n}(C(X)))$, $f_{a}(\tau)=f_{b}(\tau)$. \\

Recall the definition of the spectral profile: $f_{a}(\tau)=\mu_{\tau,a}$, we prove that $\mu_{\tau,a}=\mu_{\tau,b}$ for all $\tau\in T(M_{n}(C(X)))$.\\

Note that $\rk((a(x)-\epsilon)_+)=\rk((b(x)-\epsilon)_+)$ for every $x$ and $\epsilon\geq 0$. Hence
\begin{align}
\frac{1}{n}\rk((a(x)-\epsilon)_+)&=\mu_{\tau_{x},a}((\epsilon,1])\\&=\mu_{\tau_{x},b}((\epsilon,1])\\&=\frac{1}{n}\rk((b(x)-\epsilon)_+)    
\end{align}\label{spectral-measure-at-dirac}

Suppose $\tau\in T(M_n(C(X)))$ is arbitrary, then there exists a probability measure $\nu_{\tau}\in \text{Prob}(X)$ for which 
\begin{align}
\tau(c)=\int_X tr_{n}(c(x))d\nu_{\tau}(x)   
\end{align}

Then for any $h\in C([0,1])$, we see that by functional calculus and the Riesz-Markov-Kakutani theorem, 

\begin{align}
    \int_{[0,1]}h(t)\,d\mu_{\tau,a}(t)
&=
\tau(h(a))\\
&=
\int_X
\operatorname{tr}_n(h(a(x)))
\,d\nu_\tau(x)\\
&=
\int_X
\left(
\int_{[0,1]}
h(t)\,d\mu_{\tau_x,a}(t)
\right)
d\nu_\tau(x).
\end{align}

Now we can apply our identity about spectral measures at extreme traces \eqref{spectral-measure-at-dirac}

\begin{align}
\int_{[0,1]}h(t)\,d\mu_{\tau,a}(t)
&=
\int_X
\left(
\int_{[0,1]}
h(t)\,d\mu_{\tau_x,b}(t)
\right)
d\nu_\tau(x)\\
&=
\tau(h(b))\\
&=
\int_{[0,1]}h(t)\,d\mu_{\tau,b}(t).
\end{align}
It follows that these measures are identical, $\mu_{\tau,a}=\mu_{\tau,b}.$
Now we prove the converse, suppose that $f_a(\tau)=f_{b}(\tau)$ for all $\tau \in T(A)$. Then for any $\tau\in T(A)$ the measures $\mu_{\tau,a}=\mu_{\tau,b}$ therefore
\begin{align}
d_{\tau}((a-\epsilon)_+ )=&\lim_{m\to\infty}\int_{[0,1]}((t-\epsilon)_{+}^{1/m})d\mu_{\tau,a}    \\ =&\int_{[0,1]}\chi_{(\epsilon,1]}(t) d\mu_{\tau,a}\\
=&\int_{[0,1]}\chi_{(\epsilon,1]}(t) d\mu_{\tau,b}\\
=&\lim_{m\to\infty}\int_{[0,1]}((t-\epsilon)_+^{1/m})d\mu_{\tau,b}\\
=& d_{\tau}((b-\epsilon)_+)
\end{align}
now taking $\tau$ to be an extreme trace, the induced measure is a Dirac measure over the underlying space $X$ so the equation above becomes:

\begin{align}
    d_{\tau}((a-\epsilon)_+ )=& \lim_{m\to\infty}tr_{n}\big((a(x)-\epsilon)_+^{1/m}\big)=\frac{1}{n}\rk((a(x)-\epsilon)_+)
\end{align}
for $x\in X$

\end{proof}

\noindent

\section{Proof of Theorem~\ref{main_result_1}}

\begin{customthm}{1}

    Suppose that $A=\varinjlim (A_{i},\varphi_{i})$ is a simple, type I Villadsen algebra, each $A_{i}=M_{r_{i}}\otimes C(X^{m_{i}})$ with diagonal eigenvalue maps, negligible point evaluations, and $r_i\to \infty$. There is a subsemigroup of $Cu(A)$ denoted $\text{Rksp}(A)$ and it is $\mathbf{Cu}$- isomorphic to $\{0\}\sqcup\mathrm{LAff}(T(A))_{++}$.\end{customthm}

\begin{proof}
$\text{Rksp}(A)$ consists of the Cuntz classes of admissible realizations of functions in $\{0\}\sqcup\mathrm{LAff}(T(A))_{++}$.\\ 
\begin{align}
\text{Rksp}(A):=\{0\}\sqcup\bigg\{ x_f: x_f \text{ is an admissible realization for } f\in \text{LAff}(T(A))_{++}\bigg\}
\end{align}
It suffices to show that for every $f \in \mathrm{LAff}(T(A))_{++}$ \\

there exists a unique   
$\langle a\rangle \in $\text{Rksp}(A)$ \subseteq Cu(A)$ such that  
\begin{align}
\widehat{a} = f \quad \text{on } T(A).
\end{align}
Let $A$ be as above. Moreover, it is known that 

\begin{align}
\bigcup_{i \in \mathbb{N}}
\phi_{i,\infty}^\bullet\bigl(\mathrm{Aff}(T(A_i))_{++}\bigr)
\end{align}

is uniformly dense in $\mathrm{Aff}(T(A))_{++}$.

Let $f \in \mathrm{LAff}(T(A))_{++}$. From \cite{goodearl1986partially}, there exists an increasing sequence  
\begin{align}
f_1 < f_2 < \cdots
\end{align}  

in $\mathrm{Aff}(T(A))_{++}$ such that  

\begin{align}
f = \lim_{i \to \infty} f_i.
\end{align}

For each $f_i$, proceed as follows:\newline

 By \ref{lem:recursive-selection}, choose $(j_{n_{i}})_{i\in \mathbb{N} }$ $g_{n_i}\in \mathrm{Aff}(T(A_{j_{n_{i}}}))_{++}$ and $\eta_{n_i}\in \R$ tending to zero which satisfy
 
\begin{align}
\frac{1}{r_{j_{n_{i}}}}<\eta_{n_{i}}\\
||\phi^{\bullet}_{j_{n_{i}},\infty}(g_{n_{i}})-f_{i}||\to0\\
\phi^{\bullet}_{j_{n_{i}},\infty}(g_{n_{i}})+\eta_{n_{i}}<f_{i}\\
\phi^{\bullet}_{j_{n_{i}},j_{n_{i+1}}}(g_{n_{i}})+\eta_{n_{i}}+\eta_{n_{i+1}}<g_{n_{i+1}}
\end{align}

now we are free to apply \ref{cor:bounded_fn_rep} to obtain $a_{n_{i}}\in M_{k_n}\otimes A_{j_{n_{i}}}$ for each $g_{n_{i}}$ and \ref{lem:compatibility_along_sequence} to conclude that $\phi_{j_{n_{i}},j_{n_{i+1}}}^{(k_i)}(a_{n_{i}})\precsim a_{n_{i+1}}$ for all $(n_{i})_{i\in \mathbb{N}}$ hence $\phi_{j_{n_{i}},\infty}^{(k_i)}(a_{n_{i}})\precsim \phi_{j_{n_{i+1}},\infty}^{(k_{i+1})}(a_{n_{i+1}})$. So we find that $(\phi_{j_{n_{i}},\infty}^{(k_i)}(a_{n_{i}}))_{i\in\mathbb{N}}$ forms an increasing sequence in the sense of Cuntz subequivalence, and thus their classes have a supremum in $Cu(A)$ we will denote by $\langle a_f\rangle$, this $a_f$ satisfies $\hat a_f=f$ by our approximation above. From this construction, we denote any Cuntz class arising from $f$ by $\langle a_f\rangle$. The uniqueness of $\langle a_f\rangle $ is quickly deduced from \ref{thm:interleaving}.\\

This produces a well-defined (\ref{thm:interleaving}) map \begin{align}\Phi :\{0\}\sqcup\text{LAff}(T(A))_{++}\to \text{Rksp}(A) 
\end{align}
defined by $f \mapsto \langle a_f\rangle$, and has the inverse map \begin{align}
    \Psi:\text{Rksp}(A)\to \{0\}\sqcup\text{LAff}(T(A))_{++}
\end{align} defined by $\langle a\rangle \mapsto \big(\tau\mapsto d_{\tau}(\langle a\rangle)\big)$ the dimension function associated to $\langle a_f\rangle$. We continue by proving that $\Phi$ is a $\textbf{Cu}$-morphism. 
\begin{enumerate}
    \item $\Phi$ is clearly additive. Let $f,g\in \text{LAff}(T(A))_{++}$, then $\Phi(f+g)=\langle a_{f+g}\rangle\in \text{Rksp}(A)$. To see that $\langle a_f\rangle+\langle a_g\rangle=\langle a_{f+g}\rangle$, note that $\Phi(f)+\Phi(g)=\langle a_f\rangle+\langle a_g\rangle=\langle a_f\oplus a_g\rangle$. It is also straightforward that $ \langle a_f\oplus a_g\rangle$ is an element of $\text{Rksp}(A)$, since the direct sum of two suprema of diagonal elements is the supremum of the direct sum of the diagonal approximants \ref{def:cuntz-semigroup-axioms}.  Since $\Psi (\langle a_f\oplus a_g\rangle)=d_{\tau}(a_f\oplus a_g)=d_{\tau}(a_f)+d_{\tau}(a_g)=f+g$ and the elements in $\text{Rksp}(A)$ are determined by their rank functions via \ref{thm:interleaving}, additivity is proved.
    \item $\Phi$ preserves order, as $f\leq g$ in $\text{LAff}(T(A))$ implies that $\langle a_f\rangle\leq \langle a_g\rangle$, this follows from \ref{thm:interleaving}.
    \item $\Phi$ preserves the zero element, $\Phi(0)=\langle a_0\rangle $. Where $d_{\tau}(a_0)=0$ over the entire trace space, hence $a_0=0$ since $A$ is simple and thus every trace is faithful.
    \item $\Phi$ preserves suprema of increasing sequences. Suppose that $(f_i)$ is an increasing sequence with $\sup_i(f_i)=f$, then $\langle a_{f_i}\rangle\leq\langle a_{f_{i+1}}\rangle$ for every $i\in \mathbb{N}$. In particular, $\langle a_f\rangle$ is an upper bound for the sequence $(\langle a_{f_i}\rangle)_{i}$. Suppose now that $\langle b\rangle\in \text{Rksp}(A)$ is an upper bound for the sequence, in other words: $\langle b\rangle\geq \langle a_{f_i} \rangle$ for all $i\in \mathbb{N}$. Now take the dimension function of both sides, $d_{\tau}(\langle b\rangle)\geq d_{\tau}(\langle a_{f_i}\rangle)$ for all $i\in \mathbb{N}$. We rewrite this as $d_{\tau}(\langle b\rangle)\geq f_i$ for all $i\in \mathbb{N}$, but the supremum of this sequence is $f$, so $d_{\tau}(\langle b\rangle)\geq f$, now we apply the map to both sides to get the result.
    \item It follows that $\Phi$ preserves compact containment. More explicitly, suppose that $f\ll g$, let $(z_n)_{n\in\mathbb{N}}\in \text{Rksp}(A)$ be an increasing sequence such that $\Phi(g)\leq \sup_{n} (z_{n})=z$, then by taking the dimension function of both sides, we see that $g=d_{\tau}(\langle a_g\rangle)\leq d_{\tau}(z)$ for the increasing sequence $(d_{\tau}(z_n))_{n\in\mathbb{N}}$, since $f\ll g$, there exists $m_0\in \mathbb{N}$ satisfying $f\leq d_{\tau}(z_{m_0})$, again, taking $\Phi$ of both sides gives the result. \\
    \item To show injectivity, suppose that $f,g\in \text{LAff}_{++}(T(A),[0,\infty])$ with corresponding $a_f ,a_g\in (A\otimes \mathcal{K})_+$ satisfying $\langle a_f\rangle=\langle a_g\rangle$, then it follows immediately that their rank functions are equal since rank functions are constant along Cuntz classes.

\end{enumerate}

To finish the proof, we also note that $\Psi$ is a $\textbf{Cu}$-morphism.

\begin{enumerate}
    \item $\Psi$ is additive \ref{gardella-dim-fns}.
    \item $\Psi$ preserves order \ref{gardella-dim-fns}.
    \item $\Psi$ preserves the zero element \ref{gardella-dim-fns}.
    \item $\Psi$ preserves increasing sequences \ref{gardella-dim-fns}. 
    \item $\Psi$ preserves compact containment, suppose that $\langle a\rangle\ll\langle b\rangle$ in $\text{Rksp}(A)$ and that $(g_n)_{n\in \mathbb{N}}$ is an increasing sequence with supremum $g$ satisfying $g\geq \Psi(\langle b\rangle)$. By definition of $\Psi$, this gives us that $d_{\tau}(\langle b\rangle)\leq g$ on $T(A)$. By our construction above, find $a_{g_n}\in \text{Rksp}(A)$ with $\hat a_{g_n}=g_n$, then we have, by \ref{thm:interleaving} that $\langle b\rangle\leq \langle a_g\rangle=\sup_n \langle a_{g_n}\rangle$ in $\text{Rksp}(A)$, by hypothesis, there exists $n_0\in\mathbb{N}$ for which, $\langle a\rangle\leq \langle a_{g_{n_0}}\rangle$, now apply $\Psi$ to both sides and conclude that $\Psi(\langle a\rangle)\leq g_{n_0}$.
\end{enumerate}

\end{proof}

\section{Proof of Theorem~\ref{main_result_2}}\label{proof}

\begin{customthm}{2}
    There exists a simple, separable, AH $C^*$-algebra $A$ and two positive elements $a$ and $b$ in $A$ with spectrum equal to $[0,1]$ that are not approximately unitarily equivalent, yet $f_a(\tau) =f_b(\tau) $ for all $\tau\in T(A)$.
\end{customthm}

\begin{proof}
We use a diagonal AH construction of the same form as that in
\cite{TomsAnnals}. Choose positive integers $n_i$, with $n_1=1$, such that
\begin{equation}\label{thm2:product}
  \prod_{j=2}^{\infty}\frac{n_j}{n_j+j}>\frac12.
\end{equation}
For example, one may take $n_j=j2^{j+2}$ for $j\geq2$:
the inequality $\prod_j(1-t_j)\geq1-\sum_jt_j$ for $0\leq t_j\leq1$
shows that the resulting product is greater than $7/8$.
Set
\[
 N_i=\prod_{j=1}^i n_j,\qquad
 M_i=\prod_{j=2}^i(n_j+j),\qquad
 k_i=4M_i,\qquad X_i=[0,1]^{6N_i},
\]
where an empty product is $1$, and put $A_i=M_{k_i}(C(X_i))$.
Under the identification $X_i=X_{i-1}^{n_i}$, let
$\pi_l^i:X_i\to X_{i-1}$ be the $l$th coordinate projection. Define
\[
 \phi_{i-1}(g)=\operatorname{diag}\bigl(
 g\circ\pi_1^i,\ldots,g\circ\pi_{n_i}^i,
 g(x_1^{i-1}),\ldots,g(x_i^{i-1})\bigr),\qquad i\geq2.
\]
Choose the evaluation points so that, for each $s$ and each nonempty open
subset $U\subset X_s$, some composite map $\phi_{s,t}$ has a constant
eigenvalue map with value in $U$. Such a choice is possible by enumerating
countable bases for all the $X_s$, assigning each requirement to a distinct
later stage, and using the surjectivity of the coordinate projections to
choose an evaluation point. Consequently, every nonzero positive element
of $A_s$ has full image at some later stage, and
$A=\varinjlim(A_i,\phi_i)$ is simple. The connecting maps are unital and
injective, so $A$ is a separable unital AH algebra and we may identify
each $A_i$ with its image in $A$.

\smallskip
Let $L$ be the Hopf line bundle over $S^2$, represented in $M_4(C(S^2))$
by the projection
\[
 \xi(v)=\frac12\begin{pmatrix}
 1+v_3&v_1-\mathrm{i}v_2&0&0\\
 v_1+\mathrm{i}v_2&1-v_3&0&0\\
 0&0&0&0\\0&0&0&0
 \end{pmatrix},\qquad v=(v_1,v_2,v_3)\in S^2.
\]
Write $\theta_1=e_{11}\in M_4(\mathbb C)$ for the constant rank-one
projection. Put $c=(1/2,1/2,1/2)$ and
\[
 B=\{s\in[0,1]^3:1/8<\|s-c\|<3/8\},\qquad
 r(s)=\frac{s-c}{\|s-c\|}\quad(s\in B).
\]
Choose a continuous function $f:[0,1]^3\to[0,1]$ supported in $B$ and
equal to $1$ when $\|s-c\|=1/4$. 

Using $X_1=[0,1]^3\times[0,1]^3$, define
\begin{align*}
 a_1(s,t)=&
 \begin{cases}
 f(s)\xi(r(s)),&s\in B\\
 0,&s\notin B
 \end{cases}\\
 \qquad b_1(s,t)=&f(s)\theta_1.
\end{align*}
These are continuous positive contractions: the support of $f$ is a
compact subset of $B$, so the first formula extends continuously by zero.
At each point $(s,t)$, both matrices have eigenvalues
$\{f(s),0,0,0\}$. Since $f$ takes every value in $[0,1]$, both $a_1$ and
$b_1$ have spectrum $[0,1]$. Let
$a=\phi_{1,\infty}(a_1)$ and $b=\phi_{1,\infty}(b_1)$.
Injectivity preserves their spectra, so $\sigma(a)=\sigma(b)=[0,1]$.

\smallskip

For every $F\in C([0,1])$ and every $(s,t)\in X_1$, the normalized
matrix trace satisfies
\[
 \operatorname{tr}_4(F(a_1(s,t)))
 =\frac14F(f(s))+\frac34F(0)
 =\operatorname{tr}_4(F(b_1(s,t))).
\]
Every tracial state on $A_1$ is obtained by integrating
$\operatorname{tr}_4$ against a probability measure on $X_1$.
Thus $\tau_1(F(a_1))=\tau_1(F(b_1))$ for all $\tau_1\in T(A_1)$.
For any $\tau\in T(A)$, apply this equality to
$\tau_1=\tau\circ\phi_{1,\infty}$ to obtain
\[
 \tau(F(a))=\tau(F(b)),\qquad F\in C([0,1]).
\]
The associated probability measures are therefore equal, and hence
$f_a(\tau)=f_b(\tau)$ for every $\tau\in T(A)$.

\smallskip

Now we prove that the obstruction is present in each finite stage. For $i\geq1$, write $a_i=\phi_{1,i}(a_1)$ and
$b_i=\phi_{1,i}(b_1)$, with $\phi_{1,1}=\operatorname{id}$.
The composite $\phi_{1,i}$ has $M_i$ eigenvalue maps to $X_1$.
Exactly $N_i$ are the coordinate projections in
$X_i=X_1^{N_i}$; the remaining $P_i=M_i-N_i$ are constant maps.
Indeed, a composition is a coordinate projection precisely when all of
its factors are coordinate projections; any point evaluation makes the
composition constant. By \eqref{thm2:product},
\begin{equation}\label{thm2:rank-budget}
 \frac{N_i}{M_i}=\prod_{j=2}^i\frac{n_j}{n_j+j}>\frac12,
 \qquad P_i<N_i.
\end{equation}

Fix $t_0\in[0,1]^3$ and embed $Y_i=(S^2)^{N_i}$ in $X_i$ by
\[
 \jmath_i(v_1,\ldots,v_{N_i})
 =\bigl((c+v_1/4,t_0),\ldots,(c+v_{N_i}/4,t_0)\bigr).
\]
Let $\eta_j:Y_i\to S^2$ be the $j$th coordinate projection, and set
\[
 E_i=\bigoplus_{j=1}^{N_i}\eta_j^*L.
\]
After a fixed permutation of the matrix blocks, the restrictions
$\alpha_i=\jmath_i^*(a_i)$ and $\beta_i=\jmath_i^*(b_i)$ have the form
\begin{align*}
 \alpha_i&=\operatorname{diag}\bigl(
 \xi\circ\eta_1,\ldots,\xi\circ\eta_{N_i},C_i\bigr),\\
 \beta_i&=\operatorname{diag}\bigl(
 \theta_1,\ldots,\theta_1,D_i\bigr).
\end{align*}
Here $C_i,D_i$ are constant positive matrices, each comprising $P_i$
blocks of size $4$, with each block of rank at most one.
The scalar cutoff values in these constant blocks are retained;
in particular, $C_i$ and $D_i$ need not be projections.
At $i=1$ the constant blocks are absent.

Set $\ell_i=\operatorname{rank}(C_i)\leq P_i<N_i$.
Because $C_i$ is constant, its support is a constant projection.
It follows that the support projection $R_i$ of $\alpha_i$ is continuous
and has associated vector bundle
$E_i\oplus\theta_{\ell_i}$, where $\theta_m$ denotes the trivial
rank-$m$ bundle over $Y_i$. Also set
\[
 q_i=\operatorname{diag}(\underbrace{\theta_1,\ldots,\theta_1}_{N_i
 \text{ blocks}},0_{4P_i}).
\]
This is a constant rank-$N_i$ projection and
$q_i\beta_iq_i=q_i$.

We claim that $\theta_{N_i}$ is not isomorphic to a subbundle of
$E_i\oplus\theta_{\ell_i}$. To see this, let
$h=c_1(L)\in H^2(S^2;\mathbb Z)$, which is a generator up to sign,
and put $h_j=\eta_j^*h$. The cohomology ring is
\[
 H^*(Y_i;\mathbb Z)
 \cong\mathbb Z[h_1,\ldots,h_{N_i}]/(h_1^2,\ldots,h_{N_i}^2).
\]
The Whitney product formula gives
\begin{equation}\label{thm2:chern}
 c(E_i)=\prod_{j=1}^{N_i}(1+h_j),\qquad
 c_{N_i}(E_i)=h_1\cdots h_{N_i}\neq0.
\end{equation}
If $\theta_{N_i}$ embedded in $E_i\oplus\theta_{\ell_i}$, its
orthogonal complement would be a bundle $G_i$ of rank $\ell_i$, with
\[
 E_i\oplus\theta_{\ell_i}\cong\theta_{N_i}\oplus G_i.
\]
Consequently $c_{N_i}(E_i)=c_{N_i}(G_i)=0$, since $\ell_i<N_i$,
contradicting \eqref{thm2:chern}.

We now prove the uniform estimate
\begin{equation}\label{thm2:uniform}
 \|x a_i x^*-b_i\|\geq\frac12
 \qquad\text{for every }x\in A_i\text{ and every }i\geq1.
\end{equation}
Suppose instead that the norm is less than $1/2$, and write
$\bar x=\jmath_i^*(x)$. Since restriction and compression are
contractive, the positive element
\[
 d=q_i\bar x\alpha_i\bar x^*q_i
\]
in the corner $q_iM_{k_i}(C(Y_i))q_i$ satisfies
$\|d-q_i\|<1/2$. It is therefore invertible in that corner.
Define, with $d^{-1/2}$ taken in the corner,
\[
 v=\alpha_i^{1/2}\bar x^*q_i d^{-1/2}.
\]
Then $v^*v=q_i$ and $R_iv=v$, so $vv^*\leq R_i$.
Thus $v$ gives a vector-bundle embedding
$\theta_{N_i}\hookrightarrow E_i\oplus\theta_{\ell_i}$,
contradicting the preceding paragraph. This proves
\eqref{thm2:uniform}.

\smallskip

The union of the $A_i$ is dense in $A$, and the inclusions are isometric.
Hence \eqref{thm2:uniform}, together with continuity of the map
$x\mapsto xax^*$, implies
\[
 \|xax^*-b\|\geq\frac12\qquad(x\in A).
\]
In particular,
\[
 d_U(a,b)=\inf_{u\in\mathcal U(A)}\|uau^*-b\|\geq\frac12.
\]
Thus $a$ and $b$ are not approximately unitarily equivalent, although
their spectral profiles agree.
\end{proof}

\section{Final Remarks}

The two main results of this paper paint a complementary picture of the Cuntz semigroup for AH algebras with diagonal maps, particularly Villadsen algebras of the first type.

Theorem~\ref{main_result_1} identifies a tractable subsemigroup of $Cu(A)$, $\text{Rksp(A)}$. This semigroup realizes the ordered part of the invariant that consists of the rank functions, it acts as a spine onto which all other non-compact elements in $Cu(A)$ collapse after tensoring with the Jiang-Su algebra. 

Theorem~\ref{main_result_2} shows that the Cuntz semigroup of $A$ contains more information than spectral data alone can detect: in the absence of strict comparison, there exist positive elements $a$ and $b$ with identical spectral profiles $f_a = f_b$ (hence $d_P(a,b) = 0$) that are nonetheless not approximately unitarily equivalent ($d_U(a,b) > 0$). The gap between $d_P$ and $d_U$ is therefore non-trivial, and its origin is topological; it is witnessed by the presence of non-trivial vector bundles over the seed spaces of the Villadsen algebra.

\bibliography{bibliography}
\bibliographystyle{abbrv}
\end{document}